\documentclass[10pt]{amsart}
\usepackage{amsmath,amsthm,amscd,amsfonts,amssymb,graphicx,color,MnSymbol,enumerate,tikz-cd}
\usepackage{hyperref,cleveref}
\usepackage{xcolor}
\usepackage[labelformat=empty]{caption,subcaption}

\newtheorem{theorem}{Theorem}[section]
\newtheorem{lemma}[theorem]{Lemma}
\newtheorem{proposition}[theorem]{Proposition}
\newtheorem{corollary}[theorem]{Corollary}
\theoremstyle{definition}
\newtheorem{definition}[theorem]{Definition}
\newtheorem{example}[theorem]{Example}

\theoremstyle{remark}
\newtheorem{remark}[theorem]{Remark}

\numberwithin{equation}{section}

\allowdisplaybreaks

\begin{document}

\title[Abstract Key Polynomials and Distinguished Pairs]{Abstract Key Polynomials and Distinguished Pairs}
\author[Sneha Mavi]{Sneha Mavi}
\address{Department of Mathematics\\ University of Delhi\\ New Delhi-110007, India.}
\email{mavisneha@gmail.com}
\author[Anuj Bishnoi]{Anuj Bishnoi$^\ast$}
\address{Department of Mathematics\\  University of Delhi \\  New Delhi-110007, India.}
\email{abishnoi@maths.du.ac.in}

\begin{abstract}
In this paper, for a henselian valued field $(K,v)$  of arbitrary rank and   an extension $w$ of $v$ to $K(X),$  we use abstract key polynomials for $w$ to obtain distinguished pairs and   saturated distinguished chains. 
\end{abstract}
\thanks{$^\ast$Corresponding author, E-mail address: abishnoi@maths.du.ac.in}

\maketitle

\section*{Introduction }
Let $(K,v)$ be a henselian valued field of arbitrary rank, $\overline{K}$ be a fixed algebraic closure of $K$ and $w$ an extension of $v$ to $K(X).$ In 1995, Popescu and Zaharescu \cite{PZ} introduced the concept of distinguished pairs and saturated distinguished chains over a local field. Later, this work was generalized by Aghigh et al.\  (cf.\ \cite{AK1}-\cite{AN2}) to henselian valued fields. The notion of saturated distinguished chains is used to find results about irreducible polynomials and to obtain various invariants associated with elements of $\overline{K}.$ Recently Jakhar and Sangwan in \cite{JS}, established a connection between distinguished pairs and key polynomials over a residually transcendental extension of $v.$ In this paper, we associate distinguished pairs with abstract key polynomials over an extension $w$ of $v$ to $K(X).$ Moreover, it will be shown that abstract key polynomials for $w$  leads to  saturated distinguished chains. In \cite{PP} a characterization of key polynomials over a residually transcendental extension  was given by using liftings of irreducible polynomials, in  this paper, we relate  abstract key polynomials for $w$ with liftings of some monic polynomials over the residue field of some truncation of $w.$  As in \cite{JN1},  we  characterize those abstract key polynomials for  $w$ which are also key polynomials for some truncation of $w.$ 

 Abstract key polynomials was first introduced by Herrera et al.\  in \cite{DSM}, as an alternative definition of key polynomials. The concept of  key polynomials was first introduced by Maclane  \cite{M} in order to understand the extensions of discrete rank $1$ valuations of $K$ to $K(X).$  This work was generalized to arbitrary valuations by Vaqui\'e  in $ 2007$ (cf.\ \cite{V}). The main difference between these two types of polynomials is that a key polynomial for a given valuation $w$ allows us to augment $w,$ while an abstract key polynomial for  $w$  allows us to truncate $w.$  For more relation between key polynomials and abstract key polynomials see \cite{DSM}, \cite{NS} and \cite{JN1}.
 
To state the main results of the paper,  we first recall some notations and definitions. Next, we give some preliminary results which will be used in the proofs of the main results.
\section{Notations, Definitions and Main Results}

Throughout the paper, $(K,v)$ denote a   valued field of arbitrary rank  with value group $\Gamma_ v,$ valuation ring $O_v,$   residue field  $k_{v}$ and $\bar{v}$ is the  extension of $v$ to a fixed algebraic closure $\overline{K}$ of $K$ with value group $\Gamma_{\bar{v}}$ and residue field $k_{\bar{v}}.$  If $L$ is an extension field of $K,$ then an extension $v_L$ of $v$  to $L$ is called \emph{residually transcendental} (abbreviated as r.t.) if the corresponding residue field extension $k_{v_L}/ k_v$ is transcendental. Let $w$ be  an extension of $v$ to the simple transcendental extension $K(X)$ of $K$ with value group $\Gamma_{w}$ and residue field $k_{w}.$ An extension $\overline{w}$ of $w$ to $\overline{K}(X)$ which is also an extension of $\bar{v}$ is called a \emph{common extension} of $w$ and $\bar{v}.$
 A common extension $\overline{w}$  is r.t.\ if and only if $w$ is an r.t.\  extension of $v$  (cf.\  Proposition 2.1, \cite{APZ2}).\\

Let $\overline{w}$ be a  common extension of  $w$ and $\bar{v}$ to $\overline{K}(X).$ Then
for a pair $(\alpha,\delta)\in\overline{K}\times\Gamma_{\overline{w}},$ the map $\overline{w}_{\alpha,\delta}: \overline{K}[X]\longrightarrow \Gamma_{\overline{w}},$ given by $$\overline{w}_{\alpha,\delta}\left(\sum_{i\geq 0} c_i (X-\alpha)^i\right):=\min_{i\geq 0}\{\bar{v}(c_i)+i\delta\}, \, c_i\in\overline{K},$$
is a valuation on $\overline{K}[X]$ and can be uniquely extended to $\overline{K}(X)$  (cf. Theorem 2.2.1, \cite{En-Pr}). Such a valuation is said to be defined by $\min,\, v,\, \alpha$ and $\delta.$ 

\begin{definition}
A pair $(\alpha,\delta)$ in $\overline{K}\times \Gamma_{\bar{v}}$ is called a $(K,v)$-\emph{minimal pair}  if whenever $\beta$ in $\overline{K},$ satisfying $\bar{v}(\alpha-\beta)\geq\delta,$ then $\deg\beta\geq\deg\alpha,$ where by $\deg\alpha$ we mean the degree of the extension $K(\alpha)/K.$ 
\end{definition}

\begin{definition}\label{1.2.2}
For a common extension $\overline{w}$ of $w$ and $\bar{v},$ we say that a pair $(\alpha,\delta)$ in $\overline{K}\times \Gamma_{\overline{w}}$ is a \emph{minimal pair for $w$} if whenever $\beta$ in $\overline{K},$ satisfying $\bar{v}(\alpha-\beta)\geq\delta,$ then $\deg\beta\geq\deg\alpha.$  
\end{definition}
 
 \begin{remark}\label{1.2.1}
(i)  If  $\Gamma_{\overline{w}}=\Gamma_{\bar{v}},$  then the pair defined in Definition \ref{1.2.2}, is a $(K,v)$-minimal pair.\\
 (ii)  A pair $(\alpha,\delta)\in\overline{K}\times\Gamma_{\overline{w}}$ is called minimal pair of definition for $\overline{w}$ if $(\alpha,\delta)$ is a minimal pair for $w$ and $\overline{w}=\overline{w}_{\alpha,\delta}.$
 \end{remark}
 \begin{proposition}[Proposition 2, \cite{A-P}]
If  $\overline{w}$ is an r.t.\ extension of $\bar{v}$ to $\overline{K}(X),$
then there exist a  $(K,v)$-minimal pair $(\alpha, \delta)$  in $\overline{K}\times \Gamma_{\bar{v}}$ such that $\overline{w}=\overline{w}_{\alpha,\delta}.$
 \end{proposition}
The following result gives some properties of an r.t.\ extension and  is known as the representation theorem for  residually transcendental extensions.
\begin{theorem}[Theorem 2.1, \cite{APZ1}]\label{1.2.3}
Let $(K,v)$ and $(\overline{K},\bar{v})$ be as before and $w$  an r.t.\ extension of $v$ to $K(X).$ Let $\overline{w}$ be a common extension of $w$ and $\bar{v}$ to $\overline{K}(X)$ with  minimal pair of definition $(\alpha,\delta)\in\overline{K}\times \Gamma_{\bar{v}}.$ Let $f$ be the minimal polynomial of $\alpha$ over $K$ of degree  $n$ with $w(f)=\gamma,$ and   $v_1$  the restriction of $\bar{v}$ to $K(\alpha).$ Then the following hold:

 \begin{enumerate}[(i)]
\item For a non-zero polynomial  $p$  in $K[X]$ of degree less than  $n,$  the $\overline{w}$-residue\footnote[1]{$\overline{w}$-residue is the image under the canonical homomorphism from the valuation ring of $\overline{w}$ onto its residue field.} of $\frac{p}{p(\alpha)}$ equals $1,$ i.e.,  $w(p)=v_1 (p(\alpha)).$
 \item For any  polynomial $g$ in $K[X]$ with  $f$-expansion $\displaystyle\sum_{i\geq 0}g_i f^i,$ $\deg g_i< n,$ we have 
  $$w(g)=\min_{i\geq 0}\{v_1 (g_i(\alpha))+i\gamma \}.$$
  \item Let $e$ be the smallest positive integer such that $e\gamma\in \Gamma_{v_1},$ then there exist a polynomial $h \in K[X]$ of  degree less than $n,$ such that $w(h)=v_1(h(\alpha))=e\gamma,$    the $w$-residue, $ r^{*},$ of $f^{e}/h$ is transcendental over $k_{v_1}$ and $k_{w}=k_{v_1}(r^*).$ 

 \end{enumerate}
\end{theorem}
Using the canonical homomorphism from the valuation ring $O_v$ of $v$ onto its residue field $k_v,$ we can lift any monic polynomial $X^n+\overline{a_{n-1}}X^{n-1}+\cdots+\overline{a_0}$ with coefficients in $k_v$ to obtain a monic polynomial with coefficients in $O_v.$ In 1995, Popescu and Zaharescu \cite{PZ} extended this notion using $(K,v)$-minimal pairs as follows. 

\begin{definition}\label{1.2.11}
Let $w$ be an r.t.\ extension of $v$ to $K(X).$ With notations and hypothesis   as in Theorem \ref{1.2.3}, a monic  polynomial $F$ in $K[X]$ is said to be a lifting of a monic polynomial $G(Y),$ in an indeterminate $Y=r^*$ over $k_{v_1}$ having degree $m\geq 1,$ with respect to $(\alpha,\delta)$ (or with respect to $w$), if the following conditions are satisfied:
\begin{enumerate}[(i)]
\item $\deg F=emn,$
\item $w(F)=w(h^{m})=em\gamma,$
\item $w$-residue of $\frac{F}{h^m}$ is equal to $G(Y).$ 
\end{enumerate}
\end{definition}
 The lifting $F$ of $G(Y)$ is called trivial if  $\deg F=\deg f,$ i.e., $\deg G(Y)=1$ and $\gamma=w(f)\in\Gamma_{v_1},$ where $f$ is the minimal polynomial of $\alpha$ over $K.$
 
 To see that any monic  irreducible  polynomial over $K$ is a non-trivial lifting of some   monic irreducible polynomial over the residue field of $w,$ we use the notion of distinguished pairs which was introduced by Popescu and Zaharescu \cite{PZ}, for local fields in 1995 and was later generalized to arbitrary henselian valued fields (cf. \cite{AK1} and \cite{AN1}).

\begin{definition}
A pair $(\theta,\alpha)$ of elements of $\overline{K}$ is called a $(K,v)$-\emph{distinguished pair}  if the following conditions are satisfied:
\begin{enumerate}[(i)]
\item $\deg\theta>\deg\alpha,$
\item $\bar{v}(\theta-\alpha)=\max\{\bar{v}(\theta-\beta)\mid \beta\in\overline{K},\, \deg\beta<\deg\theta \},$
\item If $\eta\in\overline{K}$ be such that $\deg\eta<\deg\alpha,$ then $\bar{v}(\theta-\eta)<\bar{v}(\theta-\alpha).$
\end{enumerate}
\end{definition}
Clearly (iii) implies that $(\alpha,\bar{v}(\theta-\alpha))$ is a $(K,v)$-minimal pair. Also for any two irreducible polynomials $f$ and $g$ over $K,$ we call $(g,f)$ a distinguished pair, if there exist a root $\theta$ of $g$ and a root $\alpha$ of $f$ such that $(\theta,\alpha)$ is a $(K,v)$-distinguished pair.\\
Distinguished pairs give rise to distinguished chains in a natural manner. A chain $\theta=\theta_r,\theta_{r-1},\ldots,\theta_0$ of elements of $\overline{K}$ is called a \emph{saturated distinguished chain}  for $\theta$  of length $r,$ if $(\theta_{i+1},\theta_{i})$ is a $(K,v)$-distinguished pair for $0\leq i\leq r-1$ and $\theta_0\in K.$ 

 For a polynomial $f$ belonging to $K[X]$ and a positive integer $b,$ let $\partial_b f:=\frac{1}{b!}\frac{\partial^b f}{\partial X^b}$ denote the $b$-th formal derivative of $f.$ Then for an extension $w$ of $v$ to $K(X)$ let
$$\epsilon(f):=\max_{b\in\mathbb{N}}\Bigg\{{\frac{w(f)-w(\partial_b f)}{b}}\Bigg\}.$$
Note that $\epsilon(f)$ belongs to the divisible hull of $\Gamma_{w}.$

\begin{definition}
A monic polynomial $Q$ in $K[X]$ is said to be an \emph{abstract key polynomial}  (abbreviated as ABKP) for $w$ if for each polynomial $f$ in $K[X]$ satisfying $\epsilon(f)\geq\epsilon(Q)$ we have $\deg f\geq \deg Q.$
\end{definition}
It is well known that an ABKP for $w$ is an irreducible polynomial (cf.\  Proposition 2.4, \cite{NS}).
\begin{definition}
For a polynomial $Q$ in $K[X]$ the $Q$-truncation of $w$ is a map $w_Q:K[X]\longrightarrow \Gamma_w$ defined by 
$$ w_Q(f):= \min_{0\leq i\leq n}\{w(f_iQ^i)\},$$
where $f=\sum_{i\geq 0} f_i Q^i,$ $\deg f_i <\deg Q$  is the $Q $-expansion of $f.$ Clearly, in view of triangle law,   $w_Q(f)\leq w(f)$ for every $f\in K[X].$
\end{definition}
 The $Q$-truncation  $w_Q$  of $w$ need not be a valuation (cf.\  Example 2.5, \cite{NS}). However,  
\begin{lemma}[Proposition 2.6,   \cite{NS}]
If $Q$ is an ABKP for $w,$ then $w_Q$ is a valuation on $K(X).$
\end{lemma}

For   an ABKP,  $Q$ in $K[X]$ for $w,$  we set 
\begin{align*}
\alpha(Q):=&\min\{\deg f\mid w_Q(f)<w(f)\}, ~
\text{(if $w_Q=w,$ then  $\alpha(Q):=\infty$) and}\\
\psi(Q):=&\{f\in K[X]\mid f \text{ is monic, }  w_Q(f)<w(f)\,  \text{and} \deg f=\alpha(Q) \}. 
\end{align*}
Clearly $\alpha(Q)\geq \deg Q.$ Also, 
observe that $w_Q$ is a proper truncation of $w,$  (i.e.,   $w_Q<w$)  if and only if $\psi(Q)\neq \emptyset.$ Moreover, when $w_Q<w$,  then  we will show that $w_Q$ is an r.t.\ extension of $v$ to $K(X)$  (see Proposition \ref{2.1.19}). 
\begin{lemma}[Lemma 2.11, \cite{NS}] \label{1.2.6}
If $Q$ is an ABKP for $w,$ then every element $F\in\psi(Q)$ is also an ABKP for $w$ and $\epsilon(Q)<\epsilon(F).$
\end{lemma}

With notations and definitions as above,  we now state the first result of the paper which  in view of Lemma \ref{1.2.6},  relates  such  ABKPs with  liftings of polynomials over  residually transcendental extensions.
\begin{theorem}\label{1.2.7}
Let $(K,v)$ be a henselian valued field of arbitrary rank,     $w$ be an  extension of $v$ to $K(X)$  and let  $Q$ be any ABKP for $w$ such that $w_Q<w.$  Then  the following hold:
\begin{enumerate}[(i)] 
\item  A polynomial   $F$  in $K[X]$ is a trivial lifting of some monic irreducible  polynomial $G(Y)\neq Y$  in $k_{w_Q}$ with respect to $w_Q$ if and only if $F\in\psi(Q)$ and $\deg F=\deg Q.$  
\item  If  $F\in\psi(Q)$  and  $\deg F>\deg Q,$ then  $F$  is a  non-trivial lifting of some monic  polynomial $G(Y)\neq Y$ in $k_{w_Q}$ with respect to $w_Q.$ 
\end{enumerate}
\end{theorem}
Note that, the converse of Theorem \ref{1.2.7} (ii) need not be true (see Example \ref{1.2.13}).
\vspace{.01pt}
  
  The following result gives some necessary and sufficient conditions under which a pair of ABKPs is  a distinguished pair.
\begin{theorem}\label{1.2.8}
Let $(K,v)$ be a henselian valued field of arbitrary rank,   $w$ be an extension of $v$ to $K(X)$ and $Q$  an ABKP for $w$ such that $w_Q<w.$  Then for  an ABKP,  $F$ in $K[X]$  for $w,$ the following are equivalent:
\begin{enumerate}[(i)]
\item $(F,Q)$ is a distinguished pair.
\item $F\in\psi(Q)$ and $\deg F>\deg Q.$
\end{enumerate}
\end{theorem}

\vspace{2pt}
We now recall  the  definition of key polynomials which was given by Maclane in \cite{M} and  later generalized in \cite{V}.
\begin{definition}
For a valuation $w$ on $K(X)$ and polynomials $f,$ $g$ in $K[X],$ we say that
\begin{enumerate}[(i)]
\item  $f$ and $g$ are $w$-equivalent and write $f\thicksim_{w} g$ if $w(f-g)>w(f)=w(g).$
\item $g$ is $w$-divisible by $f$ or  $f$  $w$-divides $g$ (denoted by $f|_{w}g$) if there exist some polynomial $h \in K[X]$ such that $g\thicksim_{w} fh.$
\item  $f$ is $w$-irreducible, if for any $h,\, q\in K[X],$ whenever $f|_{w} hq,$ then either $f|_{w}h $ or $f|_{w}q.$
\item $f$ is $w$-minimal if for every polynomial $h\in K[X]$ whenever $f|_{w}h,$ then $\deg h\geq \deg f.$
\item Any monic polynomial $f$  satisfying (iii) and (iv) is called a  \emph{key polynomial} for $w.$
\end{enumerate}
\end{definition}
The following  results gives a relation between  distinguished pairs,  key polynomials  and ABKPs.
\begin{theorem}[Theorem 3.1, \cite{JS}]\label{2.1.16}
Let $(K,v)$ be a henselian valued of arbitrary rank and  $F$ be a key polynomial over a residually transcendental extension $w$ of $v$ to $K(X)$ defined by a $(K,v)$-minimal pair $(\alpha,\delta).$  If degree of $F$ is greater than that of the minimal polynomial of $\alpha$ over $K,$ then there exist a root $\theta$ of $F$ such that $(\theta,\alpha)$ is a $(K,v)$-distinguished pair.
\end{theorem}
\begin{theorem}[Theorem 6.1,  \cite{JN1}] \label{2.1.8}
Let $Q$ be an ABKP for $w$ and $F\in\psi(Q).$ Then $Q$ and $F$ are key polynomials for $w_Q.$ 
\end{theorem}
  Keeping in mind  Theorem \ref{2.1.16} and  Theorem \ref{2.1.8}, the following result is  an immediate consequence  of Theorem \ref{1.2.8}.
  \begin{corollary}\label{1.2.15}
  Let  $Q$ be an ABKP for $w.$ 
  Then for an ABKP, $F$ in $K[X]$ for  $w,$ the following are equivalent:
  \begin{enumerate}[(i)]
  \item $(F,Q)$ is a distinguished  pair.
  \item $F$ is a key polynomial for $w_Q$ and $\deg F>\deg Q.$
  \end{enumerate}
  \end{corollary}
  Note that the above corollary also gives the converse of Theorem \ref{2.1.8} for an extension $w$ of a henselian valued field $(K,v).$ Moreover,   Theorem \ref{1.2.8}  togther with Corollary \ref{1.2.15}   gives a characterization of those ABKPs for $w$ which are key polynomials for some truncation of $w.$
  
  \vspace{5pt}

In view of Theorem \ref{1.2.8}, the following result  gives some necessary and  sufficient conditions under which an ABKP for $w$ has a saturated distinguished chain. 
\begin{corollary}\label{1.2.10}
Let $w$ be an  extension of $v$ to $K(X)$ and $Q$ be an ABKP for $w.$  Then $Q$ has a saturated distinguished chain of ABKPs if and only if there exists ABKPs,  $Q_0, Q_1,\ldots, Q_r=Q$   for $w,$ such that $\deg Q_0=1,$ $\deg Q_{i-1}<\deg Q_{i}$ and $Q_{i}\in\psi(Q_{i-1})$ for each $i,$  $1\leq i\leq r.$
\end{corollary}
In order to prove the   existence of  such a chain of ABKPs for $w,$   we first recall the following definition from \cite{NS}.
\begin{definition}
Let $\Delta$ be an ordered set, a set $\{Q_i\}_{i\in\Delta}$ of ABKPs for a valuation $w$ of $K(X)$ is said to be complete  if for every $f\in K[X]$ there exist $i\in \Delta$ such that $w_{Q_i}(f)=w(f).$
\end{definition}
It is known that  (cf.\ Theorem 1.1, \cite{NS}), every valuation $w$ on $K[X]$ admits a complete set of ABKPs. Moreover, there is a complete set  $\{Q_i\}_{i\in\Delta}$ of ABKPs  for  $w$ having the following properties (cf.\  Remark 4.6, \cite{MMS} and proof of Theorem 1.1, \cite{NS}).
\begin{remark} \label{1.2.12}
(i) $\Delta=\bigcup_{j\in I}\Delta_j$ with $I=\{0,\ldots, N\}$ or $\mathbb{N}\cup\{0\},$ and for each $j\in I$ we have $\Delta_j=\{j\}\cup\vartheta_{j},$ where $\vartheta_j$ is an ordered set without a maximal element or is empty.\\
(ii) $Q_0=X.$\\
(iii) For all $j\in I\setminus \{0\}$ we have $j-1<i<j,$ for all $i \in\vartheta_{j-1}.$\\
(iv) All polynomials $Q_i$ with $i\in\Delta_j$ have the same degree and  have degree strictly less than the degree of the polynomial $Q_{i'}$ for every $i'\in\Delta_{j+1}.$\\
(v) For any $j\in I,$  if $\alpha(Q_{j})>\deg Q_{j},$  then   $Q_{i'}\in\psi(Q_{j}),$ for every $i' \in\Delta_{j+1}$ and if  $\alpha(Q_j )=\deg Q_{j},$   then  the set $\{w(F)\mid F\in\psi(Q_{j})\}$ does not contain a maximal element, i.e.,    $\vartheta_{j}\neq \emptyset.$ \\
(vi) For each $i<i'\in\Delta$ we have $w(Q_i)<w(Q_{i'})$ and $\epsilon(Q_i)<\epsilon(Q_{i'}).$\\
(vii) Even though the set $\{Q_i\}_{i\in\Delta}$  of ABKPs  for $w$  is not unique, the cardinality of $I$ and the degree of an abstract key polynomial $Q_i$ for each $i\in I$ are uniquely determined by $w.$\\
(viii) The ordered set $\Delta$ has a maximal element if and only if the following holds:
\begin{enumerate}[(a)]
\item the set $I=\{0,\ldots,N\}$ is finite;
\item $\Delta_N=N,$ i.e., $\vartheta_N=\emptyset.$

\end{enumerate}
\end{remark}
  With notations as in the above remark, we now state the  last result of the paper which proves  the existence of the required saturated distinguished chains.
\begin{theorem}\label{1.2.4}
Let $(K,v)$ be a henselian valued field of arbitrary rank 
and  $w$ be an extension of $v$ to $K(X).$ Assume that for a complete set $\{Q_i\}_{i\in\Delta}$ of ABKPs for $w,$    $\vartheta_j=\emptyset$  for every $j\in I.$ If $I\neq \{0\},$ then there exist $n\in I\setminus \{0\}$ such that $Q_n$ has a saturated distinguished chain.
\end{theorem}

\section{Preliminaries}
Let $(K,v),$ $(\overline{K},\bar{v})$ be as in the previous section and $w$ be an extension of $v$ to $K(X).$  Let $\overline{w}$ be a common extension of $w$ and $\bar{v}$ to $\overline{K}(X).$ In this section we give some preliminary results which will be used to prove the main results. We first recall the following definition from \cite{JN}.

\begin{definition}
For any polynomial $f$ in $K[X],$ we call a root $\alpha$ of $f$ in $\overline{K}$ an \emph{optimizing root} of $f$ if 
$$\overline{w}(X-\alpha)=\max\{\overline{w}(X-\alpha)\mid f(\alpha)=0\}=\delta(f).$$
\end{definition}

\begin{proposition}[Proposition 3.1, \cite{JN}] \label{2.1.1}
Let $f$ in $K[X]$ be a monic polynomial. Then $\delta(f)=\epsilon(f).$
\end{proposition}
Note that if $F$ and $Q$ are  two ABKPs for $w$ such that  $(\theta, \alpha)$ is a 
 $(K,v)$-distinguished pair, for  optimizing roots  $\theta$ and $\alpha$ of $F$ and $Q$ respectively, then $(F,Q)$ is a distinguished pair. The next result shows that the converse also holds.
\begin{lemma}\label{2.1.11}
Let $(K,v)$ be a henselian valued field and $(\overline{K},\bar{v})$ be as before.
Let $F$ and $Q$ be ABKPs for $w$ such that $(F,Q)$ is a distinguished pair. If $\theta$ and $\alpha$ are optimizing roots of $F$ and $Q$ respectively, then $(\theta,\alpha)$ is a $(K,v)$-distinguished pair.
\end{lemma}
\begin{proof}
To prove that $(\theta,\alpha)$ is a  $(K,v)$-distinguished pair,  it is enough to show that $\theta$ is a root of $F$ for which $\bar{v}(\theta-\alpha)=\max\{\bar{v}(\theta'-\alpha)\mid \theta' ~\text{a $K$-conjugate of $\theta$}\}.$  Since $F$ is  an ABKP for $w$ and $\deg Q<\deg F,$ so $\epsilon(Q)<\epsilon(F)$ which in view of Proposition \ref{2.1.1} implies that 
$$\overline{w}(X-\alpha)=\delta(Q)<\delta(F)=\overline{w}(X-\theta).$$
On applying strong triangle law to the  above inequality we get $\bar{v}(\theta-\alpha)=\overline{w}(X-\alpha)=\delta(Q).$  As $\alpha$ is an optimizing root of $Q,$ so
$$\overline{w}(X-\alpha')\leq \overline{w}(X-\alpha)=\bar{v}(\theta-\alpha)<\overline{w}(X-\theta),$$
for each $K$-conjugate $\alpha'$ of $\alpha,$ which  on using  strong triangle law implies that
\begin{align}
\bar{v}(\theta-\alpha')=\overline{w}(X-\alpha')\leq\overline{w}(X-\alpha)=\bar{v}(\theta-\alpha),\label{1.20}
\end{align}
for each $K$-conjugate $\alpha'$ of $\alpha.$ Since $(K,v)$ is henselian, so for any $K$-conjugate $\theta'$ of $\theta$ there exist some $K$-conjugate $\alpha'$ of $\alpha$ such that $\bar{v}(\theta'-\alpha)=\bar{v}(\theta-\alpha').$ Therefore from  (\ref{1.20}) we have that 
\begin{align*}
\bar{v}(\theta'-\alpha)=\bar{v}(\theta-\alpha')\leq\bar{v}(\theta-\alpha)
\end{align*}
for every $K$-conjugate $\theta'$ of $\theta.$
\end{proof}
The following result gives a relation between minimal pairs, optimizing roots and ABKPs.
\begin{theorem}[Theorem 1.1, \cite{JN}]\label{2.1.2}
Let $Q\in K[X]$ be a monic irreducible  polynomial and $\alpha$ be an optimizing root of $Q.$ Then $Q$ is an ABKP for $w$ if and only if $(\alpha,\delta(Q))$ is a minimal pair for $w.$ Moreover,  $(\alpha,\delta(Q))$ is a minimal pair of definition  for $w$ if and only if $Q$ is an ABKP for $w$ and $w=w_Q.$ 
\end{theorem}


Let $Q$ be any ABKP for a valuation $w$ of $K(X)$ and $w_Q$ be the $Q$-truncation of $w.$ Then for any polynomial $f\in K[X]$ with $f$-expansion $\sum\limits_{i=1}^{d} f_i Q^i,$ $\deg f_i<\deg Q$ we set 
\begin{align*}
S_Q(f):&=\{0\leq i\leq d\mid w_Q(f)=w(f_i Q^i) \}\\
\text{and}~ \delta_Q(f):&=\max S_Q(f).
\end{align*}
We now recall some basic properties of ABKPs for  $w$   (cf. \cite{NS}, \cite{JN1} and \cite{NSD}).

\begin{proposition}\label{2.1.6}
For an ABKP $Q$ for $w$ the following holds:
\begin{enumerate}[(i)]
\item The polynomial $Q$ is also an ABKP for $w_Q.$
\item Let $f,\, g\in K[X]$ be polynomials with $\deg f, \, \deg g<\deg Q.$ If $fg=pQ+r,$ for some $p,\, r\in K[X]$ with $\deg r<\deg Q,$ then $w_Q(fg)=w_Q(r)<w_Q(pQ).$
\item If $F$ is an ABKP for $w$ such that $\epsilon(Q)<\epsilon(F),$ then $w_Q(F)<w(F).$
\item For any $b\in\mathbb{N}$ and for any $f$ in $K[X]$ we have 
\begin{align}\label{1.2}
\frac{w_Q(f)-w_Q(\partial_{b} f)}{b}\leq \epsilon(Q).
\end{align}
\item If $S_Q(f)\neq \{0\},$ then equality holds in (\ref{1.2})  for some $b\in\mathbb{N}.$
\item If for some $b\in\mathbb{N},$  equality holds  in (\ref{1.2})  and $w_Q(\partial_b f)=w(\partial_b f),$ then $\epsilon(f)\geq \epsilon(Q).$
\item If $F$  is an ABKP for $w$ with $\epsilon(Q)=\epsilon(F),$ then $w_Q=w_{F}.$
\end{enumerate}
\end{proposition}

\begin{remark}\label{2.1.7}
For a given ABKP,  $Q$ for $w$ and a polynomial $f$ in $K[X],$ Proposition \ref{2.1.6} (iv) implies that 
$$ \epsilon_Q(f):=\max_{b\in\mathbb{N}}\frac{w_Q(f)-w_Q(\partial_bf)}{b}\leq \epsilon(Q)$$
and by  (v),  if $S_Q(f)\neq \{0\},$ then $\epsilon_Q(f)=\epsilon(Q).$
\end{remark}
We now prove the following result which will be used to prove the foregoing lemma. Note that the idea of the proofs of next two lemmas are motivated from Proposition 3.5 of \cite{MN}.

\begin{lemma}\label{2.1.9}
Let $Q$ be an ABKP for $w$ and let  $f$ in $K[X]$ be such that $w_Q(f)<w(f).$ If $h$ in $K[X]$ is a polynomial such that $w_Q(f)<w_Q(h),$ then $w_Q(f+h)<w(f+h).$
\end{lemma} 
\begin{proof}
By the assumptions together with strong triangle we have that 
\begin{align}\label{1.3}
w_Q(f+h)=\min\{w_Q(f), w_Q(h)\}=w_Q(f)<w(f)
\end{align}
and $w_Q(h)\leq w(h)$ implies that $w_Q(f)<w(h).$ Therefore by triangle law we get   
\begin{align}\label{1.4}
w(f+h)\geq \min\{w(f), w(h)\}>w_Q(f).
\end{align}
The lemma now follows from (\ref{1.3}) and (\ref{1.4}).
\end{proof}

\begin{lemma}\label{2.1.10}
Let $Q$ be an ABKP for $w$ and let  $F\in K[X]$ be a monic polynomial of minimal degree such that $w_Q(F)<w(F)$ (i.e., $F\in\psi(Q)).$ If
$$F=f_d Q^d+ f_{d-1}Q^{d-1}+\cdots+f_{0},\hspace{0.5mm} \deg f_i<\deg Q$$
is the  $Q$-expansion of $F,$ then $\delta_Q(F)=d$ and  $f_d=1.$
\end{lemma}
\begin{proof}
Let $\delta_Q(F)=m$ and  assume to the contrary that $m< d.$ Then 
$$w_Q(F)=w(f_m Q^m)<w(f_d Q^d)=w_{Q}(f_d Q^d).$$
Now on applying Lemma \ref{2.1.9} with $f=F$ and $h=f_d Q^d,$ we get $$ w_Q(F-f_d Q^d)< w(F-f_d Q^d).$$ But $\deg (F-f_d Q^d)<\deg F,$ which contradicts the minimality of $\deg F.$

It remains to prove that $f_d=1.$   Suppose if possible that  $f_d\neq 1.$ As $F$ is monic, so $f_d$ must be a non-constant polynomial,  and therefore $\deg F> \deg Q^d.$ Since $Q$ is irreducible and $\deg f_d<\deg Q,$ so by Bezout's identity there exist a polynomial $a\in K[X]$ of degree less than $\deg Q$ such that $af_d= q_d Q+1,$ for some $q_d$ in $K[X]$ with $\deg q_d<\deg Q.$ Now on using  Proposition \ref{2.1.6} (ii), we get
\begin{align}\label{1.5}
w_Q(q_d Q)>w_Q(a f_d)=w_Q(1)=0.
\end{align}
Since for  each $i\in\{0,1,\ldots, d-1\},$  $\deg f_i<\deg Q$ and $\deg a<\deg Q,$  so the $Q$-expansion of $af_i$ must be of the form  $$a f_i=q_i Q+r_i; ~\deg q_i,~ \deg r_i<\deg Q,$$   which again on applying Proposition \ref{2.1.6} (ii), implies that 
\begin{align}\label{1.6}
 w_Q(q_i Q)>w_Q(a f_i)=w_Q(r_i).
\end{align}
For each $i,$  $0\leq i\leq d,$ on substituting for   $af_i,$ in $aF= af_d Q^d+ af_{d-1}Q^{d-1}+\cdots+ af_0,$ we get that 
 $$ aF = q_d Q^{d+1}+ [q_{d-1}+1]Q^d+ [q_{d-2}+r_{d-1}]Q^{d-1}+\cdots+[q_0+r_1]Q+r_0,$$
 which is clearly the $Q$-expansion of $aF.$ On using  $\delta_Q(F)=d$  together with  (\ref{1.5})  $(w_Q(af_d)=0),$ we have that
 \begin{align}\label{1.7}
 \hspace{20pt} w_Q(aF)= w_Q(a)+w_Q(F)=w_Q(a)+w_Q(f_d Q^d)= w_Q(af_d)+w_Q(Q^d)=w(Q^d). 
 \end{align}
Also in view of the hypothesis and the fact that $\deg a<\deg Q,$ we have that 
\begin{align}\label{1.8}
w_Q(aF)=w_Q(a)+w_Q(F)&<w(a)+w(F)=w(aF),\nonumber\\
 \text{i.e.,} \hspace{50pt} w_Q(aF)&<w(aF).
\end{align}
On adding $w(Q^d)$ to  (\ref{1.5}),   we get
\begin{align}\label{1.9}
w_Q(q_d Q^{d+1})>w(Q^d).
\end{align}
For $i=d-1$  on adding $w(Q^{d-1})$ to (\ref{1.6}), we have 
$$w_Q(q_{d-1}Q^d)>w_Q(r_{d-1}Q^{d-1})=w_Q(a f_{d-1} Q^{d-1})=w_Q(a)+w_Q(f_{d-1} Q^{d-1}),$$
which together with the fact that 
$w_Q(F)=w(f_d Q^d)\leq w(f_{d-1}Q^{d-1})$ and  (\ref{1.7}) implies that 
\begin{align*}
w_Q(q_{d-1} Q^d)>w_Q(aF)=w(Q^d).
\end{align*}
Now from  (\ref{1.9}) and the above inequality we get
$$w_Q(q_{d-1}Q^d+ q_d Q^{d+1})\geq\min\{w_Q(q_{d-1}Q^d), w_Q(q_d Q^{d+1}) \}> w(Q^d)=w_Q(aF),$$
\begin{align*}
\text{i.e.,}~ w_Q(aF)&< w_Q(q_{d-1}Q^d+q_d Q^{d+1}).
\end{align*}
In view of (\ref{1.8}) and the above inequality, the hypothesis of Lemma \ref{2.1.9} holds for $f= aF$ and $h= q_{d-1}Q^d+ q_d Q^{d+1},$ consequently we have 
$$w_Q\left(aF-[q_{d-1}Q^d+q_d Q^{d+1}]\right)< w\left(aF-[q_{d-1}Q^d+q_d Q^{d+1}]\right),$$ 
but $\deg \left(aF-[q_{d-1}Q^d+q_d Q^{d+1}]\right)=d\deg Q<\deg F$ contradicting the choice of $ F.$
\end{proof}

In 2004, Kuhlmann \cite{FV-K}, classified all extensions $w$ of $v$ to $K(X)$ as follows.
\begin{definition}
The extension $w$ of $v$ to $K(X)$ is said to be \emph{valuation-algebraic} if $\frac{\Gamma_{w}}{\Gamma_{v}}$ is a torsion group and $k_w$ is algebraic over $k_v.$ The extension $w$ is said to be \emph{value-transcendental} if $\frac{\Gamma_w}{\Gamma_v}$ is a torsion free group and $k_w$ is algebraic over $k_v.$ 
\end{definition}

\begin{definition}
The extension $w$ of $v$ to $K(X)$ is called \emph{valuation-transcendental} if $w$ is either value-transcendental or is residually transcendental.
\end{definition}

Another characterization of valuation-transcendental extension is given by the following theorem (cf. Theorem 1.1 and Theorem 3.1, \cite{NSD}).
\begin{theorem}\label{2.1.3}
An extension $w$ of $v$ to $K(X)$  is valuation-transcendental if and only if there exists  an ABKP,  $Q$ for $w$ such that $w=w_Q.$   Moreover, if $\alpha\in\overline{K}$ is an optimizing root of $Q,$ then $\overline{w}_{x-\alpha}|_{K(X)}=w_Q=\overline{w}_{\alpha,\delta(Q)}|_{K(X)}.$
\end{theorem}
In the next lemma we  show that every ABKP for $w_Q$ has degree atmost $\deg Q.$ 
\begin{lemma}\label{1.2.14}
Let $Q$ be an ABKP for  $w.$ If $F$ is an ABKP for $w_Q,$ then $\deg F\leq\deg Q.$
\end{lemma}
\begin{proof}
Suppose to the contrary that  $\deg F>\deg Q.$ Since $F$ is an ABKP for $w_Q,$ so   $\epsilon_{Q}(Q)<\epsilon_{Q}(F).$ In view of Proposition \ref{2.1.6} (i), we have that the $Q$-truncation of $w_Q$ is again $w_Q.$ Now on
 applying Proposition \ref{2.1.6} (iii) for $w_Q,$   we get  $w_{Q}(F)<w_Q(F)$  which  is not possible.
\end{proof}
\begin{remark}
 If $w_Q=w$ for some ABKP,  $Q$ for $w,$ then in view of the above lemma, there exist no  ABKP,  $F$ for $w$ such that $(F,Q)$ is a distinguished pair. Therefore, for the existence of distinguished pair of ABKPs for $w$ we must assume that $w_Q<w.$  Moreover, we  show that this assumption leads to the fact that all proper truncations of $w$ with respect to ABKPs for $w$ are always  r.t.\ extensions.
\end{remark}
\begin{lemma}[Lemma 3.5, \cite{NSD}]\label{2.1.18}
Let $w$ be a valuation of $K(X).$ Then
for any polynomial $f$ in $K[X],$ $w(f)\in\Gamma_{\bar{v}}$ if and only if $\delta(f)\in \Gamma_{\bar{v}}.$
\end{lemma}
 The following result  gives some necessary and sufficient condition under which a truncation valuation is residually transcendental. Note that the forward part is a direct consequence of Proposition 3.4 of \cite{NSD} and the converse follows from Theorem \ref{2.1.3} and  Lemma \ref{2.1.18}.
\begin{proposition}\label{2.1.4}
Let $Q$ be an ABKP for  $w.$  Then $w(Q)\in\Gamma_{\bar{v}},$ if and only if $w_Q$ is an r.t.\  extension of $v$ to $K(X).$
\end{proposition}
For an ABKP,  $Q$ for $w,$ we now  show that the proper  $Q$-truncation of  $w$ is an r.t.\  extension.
\begin{proposition}\label{2.1.19}
Let $Q$ be an ABKP for  $w$ such that $w_Q<w,$ then $w_Q$ is an r.t.\ extension of $v$ to $K(X).$
\end{proposition}
\begin{proof}
Since $w_Q<w,$  so there exist a monic polynomial $F$ in $K[X]$ such that  $F\in\psi(Q),$ which in view of Lemma \ref{1.2.6}, implies that $\epsilon(Q)<\epsilon(F).$ Now on using Proposition \ref{2.1.1}, we get  $$\overline{w}(X-\alpha)=\delta(Q)<\delta(F)=\overline{w}(X-\theta),$$ where $\theta$ and $\alpha$ are  optimizing roots of $F$ and $Q$ respectively. On applying strong triangle law to the above inequality, we have that $\delta(Q)=\bar{v}(\theta-\alpha)\in\Gamma_{\bar{v}},$ which by Lemma \ref{2.1.18},   implies that   $w(Q)\in\Gamma_{\bar{v}}.$ The  results now follows  from Proposition \ref{2.1.4}.
\end{proof}


In the previous section, we recalled the definition of complete set $\{Q_i\}_{i\in\Delta}$ of ABKPs  for $w,$ and their properties.
Using complete set of ABKPs the following result gives another characterization of a valuation-transcendental extension.
\begin{theorem}[Theorem 5.6, \cite{MMS}]\label{2.1.13}
Let $\{Q_i \}_{i\in\Delta}$ be a complete set of ABKPs for $w.$ Then $w$ is a valuation-transcendental extension of $v$ to $K(X)$ if and only if $\Delta$ has a maximal element. Moreover, if $N$ is the maximal element of $\Delta,$ then $w=w_{Q_N}.$
\end{theorem}


\section{Proof of Theorems \ref{1.2.7}, \ref{1.2.8}, \ref{1.2.4} and Corollary  \ref{1.2.10}}

\begin{proof}[Proof of Theorem \ref{1.2.7}]
(i)  Suppose first that  $F$ is a trivial lifting of $G(Y)\neq Y$ in $k_{w_Q}$ with respect to $w_Q.$ Then we have that $\deg F=\deg Q,$ $\deg G(Y)=1$ and $w_Q(F)=w_Q(Q)=w(Q),$ i.e., $S_Q(F)\neq\{0\}.$ As $\deg\partial_{b} F<\deg Q,$ so $w_Q(\partial_{b}F)=w(\partial_{b}F)$ and therefore by Proposition \ref{2.1.6} (v) and (vi), we get that  
\begin{align}\label{1.24}
\epsilon(F)\geq\epsilon(Q).
\end{align}
Let $g$  be any polynomial in $K[X]$ with $\deg g<\deg F=\deg Q,$  and  as $Q$ is an ABKP for $w,$ so  $\epsilon(g)<\epsilon(Q)$ which in view of   (\ref{1.24}) implies that $\epsilon(g)<\epsilon(F).$ Hence $F$ is an ABKP for $w.$ 

Clearly $w_Q(F)\leq w(F).$ If $w_Q(F)<w(F),$ then $F\in\psi(Q)$ and we are done.  Otherwise,  we have that 
$$\epsilon_Q(F):=\max_{b\in\mathbb{N}} \frac{w_Q(F)-w_Q(\partial_{b}F)}{b}=\max_{b\in\mathbb{N}} \frac{w(F)-w(\partial_{b} F)}{b}=\epsilon(F).$$
 It now follows from  Remark \ref{2.1.7}, and the above equality  that $\epsilon(F)\leq\epsilon(Q)$ which together with  (\ref{1.24}) implies  that 
$\epsilon(Q)=\epsilon(F)$ and  hence
  $w_Q=w_F,$ in view of Proposition \ref{2.1.6} (vii).  Since $F$ is a trivial lifting of $G(Y)$ with respect to $w_Q,$ so by Definition \ref{1.2.11} for $w_Q,$   $Y$ is the $w_Q$-residue of $\frac{Q}{h},$ where $h\in K[X]$ is a polynomial as in Theorem \ref{1.2.3} (iii), and $w_Q(F)=w_Q(Q)=w_Q(h).$ As $w_Q=w_F,$ so the  $w_Q$-residue of $\frac{Q}{h}$ is same as the $w_F$-residue of $\frac{F}{h},$ i.e., $Y$ is equal to the $w_F$-residue of $\frac{F}{h},$ which  implies that $G(Y)=Y$ contradicting the fact that $G(Y)\neq Y.$
  
\bigskip
Conversely, assume that $F\in \psi(Q)$ and $\deg F=\deg Q.$ Then
 as $F,~Q$ are both monic polynomials, so $$F=Q+(F-Q)$$ is the $Q$-expansion of $F.$ Therefore $w_Q(F)=\min\{ w(Q), w(F-Q)\},$ which in view of the hypothesis implies that  $$w_Q(F)=\min\{ w(Q), w(F-Q)\}<w(F).$$
If $\min\{ w(Q), w(F-Q)\}=w(F-Q),$ then  by strong triangle law we have  $w(Q)=w(F-Q),$ otherwise  again, on using  strong triangle law,  we  get   $w(F-Q)=w(Q).$ So in either case we have that
\begin{align}\label{1.12}
w_Q (F)=w(Q)=w(F-Q).
\end{align}
As $w_Q<w,$ so by Proposition \ref{2.1.19}, $w_Q$ is an r.t.\ extension of $v$ to $K(X).$ Let $\alpha\in\overline{K}$ be an optimizing root of $Q,$ then from Theorem \ref{2.1.3},  the common extension of $w_Q$ and $\bar{v}$ to $\overline{K}(X)$ is given by $\overline{w}_{X-\alpha}=\overline{w}_{\alpha,\delta(Q)},$ where $\delta(Q)=\overline{w}(X-\alpha)=\overline{w}_{X-\alpha}(X-\alpha),$ for  a common extension $\overline{w}$ of $w$ and $\bar{v}$  to $\overline{K}(X).$ If we denote  $F-Q$ by $h$ in the $Q$-expansion of $F,$ then $\deg h<\deg Q$ and hence $w_Q(h)=w(h),$ which in view of  (\ref{1.12}) implies that
\begin{align}\label{1.29}
w_Q(Q)=w_Q(h).
\end{align}
Let $v_1$ be the restriction of $\bar{v}$ to $K(\alpha),$ then by Theorem \ref{1.2.3} (i), for the r.t.\  extensiom $w_Q,$ (\ref{1.29}) becomes  
$$w_Q(Q)= v_1(h(\alpha)), ~\text{i.e.,}~ w_Q(Q)\in\Gamma_{v_1}.$$
Now  Theorem \ref{1.2.3} (iii)   together with the $Q$-expansion of $F$ implies that $e=m=1$ and hence from (\ref{1.29}),  the $w_Q$-residue, $\left( \frac{Q}{h}\right)^* =Y$ (say), of  $\frac{Q}{h}$   is transcendental over the  residue field of $ v_1.$ Thus $F$ is the lifting of $$\left( \frac{F}{h}\right) ^*=\left( \frac{Q}{h}\right)^*+1=Y+1\in k_{v_1}[Y].$$\\

\noindent(ii) As argued above,  $w_Q$ is an r.t.\ extension of $v$ to $K(X)$  and $\overline{w}_{\alpha, \delta(Q)}|_{K(X)}=w_Q,$ for an optimizing root $\alpha$ of $Q.$
 Since $F\in \psi(Q),$  so by Lemma \ref{1.2.6},  $F$ is  an ABKP for $w$ and $\epsilon(Q)<\epsilon(F)$ which in view of Proposition \ref{2.1.1} implies that $\delta(Q)<\delta(F).$ If  $\theta$ is an optimizing root of $F,$  then 
\begin{align*}
\overline{w}(X-\alpha)=\delta(Q)<\delta(F)=\overline{w}(X-\theta),
\end{align*}
and from  the definition of $\delta(Q),$ we have that 
$$\overline{w}(X-\alpha')\leq\overline{w}(X-\alpha)<\overline{w}(X-\theta),$$
for each  $K$-conjugate $\alpha'$ of $\alpha.$ Now on applying strong triangle law to the above inequality we get 
\begin{align}\label{1.14}
\bar{v}(\theta-\alpha')=\overline{w}(X-\alpha')\leq\overline{w}(X-\alpha)=\delta(Q),
\end{align}
for each $K$-conjugate $\alpha'$ of $\alpha.$ Since $(K,v)$ is henselian, so for any $K$-conjugate $\theta'$ of $\theta$ there exist a  $K$-conjugate $\alpha'$ of $\alpha$ such that $\bar{v}(\theta'-\alpha)=\bar{v}(\theta-\alpha').$ Therefore from (\ref{1.14}), we have that 
\begin{align}\label{1.15}
\bar{v}(\theta'-\alpha)=\bar{v}(\theta-\alpha')\leq\delta(Q),
\end{align}
for each $K$-conjugate $\theta'$ of $\theta.$ 
Again as $F\in\psi(Q),$ so by Lemma \ref{2.1.10}  the $Q$-expansion of $F$ is of the form
 $$F=Q^d+f_{d-1}Q^{d-1}+\cdots+f_0,$$
 where $\deg f_i<\deg Q$ for every $0\leq i\leq d-1$ and  
 \begin{align}\label{1.16}
  w_Q(F)=w(Q^d).
 \end{align}
 Moreover, as $\deg F>\deg Q,$ so $d>1.$
 We now show that $w_Q(F)=v_1(f_0(\alpha)),$ where $v_1$  is the restriction of $\bar{v}$ to $K(\alpha).$
Let $F=\prod_{\theta'}(X-\theta')$ be the factorization of $F$ over $\overline{K},$ where the product runs over all $K$-conjugates $\theta'$ of $\theta,$ then
\begin{align}\label{1.17}
w_Q(F)&= \sum_{\theta'} \overline{w}_{X-\alpha}(X-\theta')\nonumber\\
&=\sum_{\theta'}\min\{\overline{w}(X-\alpha),\bar{v}(\theta'-\alpha)\}\nonumber\\
&=\sum_{\theta'}\min\{\delta(Q),\bar{v}(\theta'-\alpha)\}\nonumber\\
&=\sum_{\theta'}\bar{v}(\theta'-\alpha) ~\hspace{100mm}~(\text{by (\ref{1.15}))}\nonumber\\
w_Q(F)&=v_1 (F(\alpha))=v_1(f_0(\alpha)).
\end{align}
The above equation together with  (\ref{1.16}) gives 
$$dw_Q(Q)=w(Q^d)=w_Q(F)=v_1(f_0(\alpha))\in\Gamma_{v_1},$$
which in view of  Theorem \ref{1.2.3} (iii) (for the r.t.\ extension $w_Q$)  implies that,  if $e$ is the smallest positive integer such that $ew_Q(Q)\in\Gamma_{v_1},$ then $e|d,$ i.e., $d=em,$ for some $m\in\mathbb{N}$ and there exist a polynomial $h\in K[X]$ with $\deg h<\deg Q$ such that  $w_Q(h)=ew_Q(Q)=ew(Q).$
  Therefore, on using  (\ref{1.16}), we get  that 
\begin{align}\label{1.13}
 w_Q(F)=emw(Q)=mw_Q(h)=w_Q(h^m).
 \end{align}
 
 From the $Q$-expansion of $F,$ we have that  $w_Q(F)=\min_{0\leq i\leq em}\{w(f_{em-i}Q^{em-i})\},$ which together with (\ref{1.13}) implies that
 \begin{align}\label{1.25}
  w(f_{em-i}Q^{em-i})\geq em w(Q)=w_Q(h^m),~\text{
for each $i,$ $0\leq i\leq em.$}
 \end{align}
Clearly,  from (\ref{1.13})  equality holds in (\ref{1.25}) for $i=0.$  For $i=em,$ as $\deg f_0< \deg Q,$ so by Theorem \ref{1.2.3} (i) we get  that $w(f_0)=w_Q(f_0)=v_1 (f_0(\alpha)),$ which in view of (\ref{1.17}) and (\ref{1.13}) implies that $$w_Q(f_0)=v_1(f_0(\alpha))=w_Q(h^ m),$$ i.e., equality holds  in (\ref{1.25}) for $i=em.$
Now suppose that  $w(f_{em-i}Q^{em-i})=em w(Q),$ for some $i\in\{1,\ldots,em-1\},$  i.e.,  $w(f_{em-i})=iw(Q).$   As $\deg f_{em-i}<\deg Q,$ so $$w(f_{em-i})=w_Q(f_{em-i})=v_1(f_{em-i}(\alpha))\in\Gamma_{v_1},$$ i.e., $iw(Q)\in\Gamma_{v_1},$ hence by the choice of $e,$ $e| i.$ 

From what we have shown above, it follows that  $w_Q\left( \frac{f_{em-i} Q^{em-i}}{h^m}\right) >0$ for all $i$ not divisible by $e,$  i.e.,  the $w_Q$-residue, $\left( \frac{f_{em-i} Q^{em-i}}{h^m}\right)^*=0$ for all $i$ not divisible by $e$ and   $\left( \frac{f_{em-i} Q^{em-i}}{h^m}\right)^*\neq0$ implies that   $e| i.$  So from the $Q$-expansion of $F,$ we get 
\begin{align*}
\left(\frac{F}{h^m} \right)^* &=\left(\frac{Q^{em}}{h^m} \right)^*+\left(\frac{f_{e(m-1)}(\alpha)Q^{e(m-1)}}{h^m} \right)^*+\cdots+\left(\frac{f_0(\alpha)}{h(\alpha)^m} \right)^*\\
&=Y^m+\left(\frac{f_{e(m-1)}(\alpha)}{h(\alpha)}\right)^* Y^{m-1}+\cdots+ \left(\frac{f_0(\alpha)}{h(\alpha)^m} \right)^*=G(Y) \neq Y,
\end{align*} 
where $Y=\left(\frac{Q^e}{h} \right)^*.$   Thus, $F$ is the lifting of $G(Y)\in k_{v_1}[Y]$ with respect to $w_Q.$
\end{proof}
 \begin{remark}
 Let $w$ be an extension of $v$ to $K(X)$   and $Q$ be  an ABKP for $w$ such that $w_Q<w.$ If  $F$ is a trivial lifting of some monic irreducible polynomial over the residue field  of $w_Q,$ then  from  Theorem \ref{1.2.7} (i),   $F$ is an ABKP for $w$ and $F\in\psi(Q)$  which together with   Theorem \ref{2.1.8},  implies that  $F$ is a key polynomial for $w_Q,$ i.e., trivial liftings with respect to $w_Q$ are not only ABKPs for $w$ but also key polynomials for $w_Q.$   However, if $F$ is a non-trivial lifting of some monic polynomial over the residue field of $w_Q,$ then $F$ need not belong  to $\psi(Q).$  
 \end{remark}
 \begin{example}\label{1.2.13}
 Let $\mathbb{Q}_3$ be  the   field of $3$-adic numbers with $3$-adic valuation $v_3.$  Let $f=X^2+1$ be a polynomial over $\mathbb{Q}_3$ having a  root   $\alpha.$  Since $\bar{f}$ is irreducible over the residue field of $v_3,$ so for $\delta=\frac{1}{2},$  $(\alpha,\delta)$ is a $(\mathbb{Q}_3,v_3)$-minimal pair (see. \cite{KS},  p.\ 2649). If $\overline{w}$ is the r.t.\  extension of $\bar{v}_3$ to $\overline{\mathbb{Q}}_3(X)$ defined by the $(\mathbb{Q}_3,v_3)$-minimal pair  $(\alpha,\delta)$ and  $w$ its  restriction  to $\mathbb{Q}_3(X),$  then it can be easily shown that $\overline{w}(X^2+1)=\frac{1}{2}.$    Let $Q=X,$ then $Q$ is an ABKP for $w$ and for the $Q$-truncation of $w$ we have 
  $$w_Q (X^2+1)=\min\{2w(X), v_3 (1)\}=\min\{0,0\}=0<\frac{1}{2}=w(X^2+1),$$ i.e., $w_Q<w.$ Consider now the polynomial  $F=X^2+2$ in $\mathbb{Q}_3[X].$ Then $F$   is a non-trivial lifting of $Y^2+2$ with respect to $w_Q,$ where $Y$ is the $w_Q$-residue of $Q.$  But
$$w_Q(X^2+2)=\min\{2w(X), v_3(2)\}=\min\{0,0\}=0$$ and 
$$w(X^2+2)=\overline{w}(X^2+2)=\min\{\overline{w}(X^2+1), v_3(1)\}=\min\{1/2, 0\}=0,$$ i.e,   $w_Q(F)=w(F),$ hence  $F\notin\psi(Q).$
 \end{example}
\begin{proof}[Proof of Theorem \ref{1.2.8}]
$(i)\Rightarrow (ii)$ Assume  to the contrary that $F\notin \psi(Q).$ Since  $w_Q<w,$ so  there exist a monic polynomial $g$ in $K[X]$ such that $g\in\psi(Q),$ which in view of  Lemma \ref{1.2.6},  implies that $g$ is an ABKP for $w$ and $\epsilon(Q)<\epsilon(g).$  If $\alpha$ and $\beta$ are optimizing roots of $Q$ and $g$ respectively, then
  from  Proposition \ref{2.1.1}, we have that  $$\overline{w}(X-\alpha)=\delta(Q)<\delta(g)=\overline{w}(X-\beta)$$ which on using strong triangle law gives 
  \begin{align}\label{1.18}
\bar{v}(\beta-\alpha)=\overline{w}(X-\alpha)<\overline{w}(X-\beta) .
  \end{align}
As $(F,Q)$ is a distinguished pair of ABKPs for $w,$  so $\deg Q<\deg F$ implies that  $\epsilon(Q)<\epsilon(F).$ Hence by  Proposition \ref{2.1.6} (iii), we have  $w_Q(F)<w(F),$  which together with the assumption  $F\notin\psi(Q)$ and the fact that $g\in \psi(Q)$  implies   $\deg F>\deg g.$ Again as $F$ is an ABKP for $w,$ so $\epsilon(g)<\epsilon(F)$ and hence from  Proposition \ref{2.1.1},  we get
$$\overline{w}(X-\beta)=\delta(g)<\delta(F)=\overline{w}(X-\theta),$$
for some optimizing root $\theta$   of $F.$  On applying strong triangle law to the above inequality we have that
$$\bar{v}(\theta-\beta)=\overline{w}(X-\beta)<\overline{w}(X-\theta).$$
The above inequality together with (\ref{1.18})   and strong triangle law  now gives
\begin{align}\label{1.19}
\bar{v}(\theta-\alpha)=\bar{v}(\beta-\alpha)<\bar{v}(\theta-\beta).
\end{align}
As $(F,Q)$ is a distinguished pair, so by Lemma \ref{2.1.11},  $(\theta,\alpha)$  is a $(K,v)$-distinguished pair and therefore  
$\deg\beta<\deg\theta,$ implies  $\bar{v}(\theta-\beta)\leq\bar{v}(\theta-\alpha)$ which contradicts  (\ref{1.19}). Thus our assumption is false.
\vspace{5pt}

$(ii)\Rightarrow(i)$ Let $\theta$ and $\alpha$ be optimizing roots of $F$ and $Q$ respectively. In order to prove that $(F,Q)$ is a distinguished pair, we show that $(\theta,\alpha)$ is a $(K,v)$-distinguished pair. Let $\beta\in\overline{K}$ be such that $\deg\beta<\deg\theta=\deg F$ and let $g$ be the minimal polynomial of $\beta$ over $K.$   As $F$ is an ABKP for $w,$ so $\epsilon(g)<\epsilon(F)$ and  therefore by definition of $\delta(g)$ and  Proposition \ref{2.1.1}, we have that
$$\overline{w}(X-\beta)\leq\delta(g)<\delta(F)=\overline{w}(X-\theta).$$
On applying strong triangle law to the above inequality we get that
\begin{align}\label{1.21}
\bar{v}(\theta-\beta)=\overline{w}(X-\beta)\leq\delta(g)<\overline{w}(X-\theta).
\end{align}
Since  $F\in\psi(Q)$ and  $\deg g<\deg F,$ so  $w_Q(g)=w(g)$ and $w_Q(\partial_{b} g)=w(\partial_{b} g)$ which in view of Proposition \ref{2.1.6} (iv), implies that
\begin{align*}
\frac{w(g)-w(\partial_{b} g)}{b}&=\frac{w_Q(g)-w_Q(\partial_{b}g)}{b}
\leq\epsilon(Q),~\hspace{10mm}\forall b\in\mathbb{N},\\
  \text{i.e.,}~\hspace{10mm} \epsilon(g)&\leq\epsilon(Q).
 \end{align*}
On using (\ref{1.21}),  Proposition \ref{2.1.1}, together with  the above inequality we get that
\begin{align}\label{1.30}
\bar{v}(\theta-\beta)=\overline{w}(X-\beta)\leq\delta(g)\leq\delta(Q)=\overline{w}(X-\alpha).
\end{align}
Again as $F\in\psi(Q),$ so from Lemma \ref{1.2.6}, $\epsilon(Q)<\epsilon(F)$ which by Proposition \ref{2.1.1},  and (\ref{1.30}) implies that 
$$\bar{v}(\theta-\beta)\leq \overline{w}(X-\alpha)=\delta(Q)<\delta(F)=\overline{w}(X-\theta).$$
Now on using strong triangle law we have 
 $\overline{w}(X-\alpha)=\delta(Q)=\bar{v}(\theta-\alpha),$ and hence
  $$\bar{v}(\theta-\beta)\leq\bar{v}(\theta-\alpha).$$
 Finally for any $\eta\in\overline{K}$ with $\deg \eta<\deg \alpha,$  we  need to show that $\bar{v}(\theta-\eta)<\bar{v}(\theta-\alpha).$ Since $Q$ is an ABKP for $w,$ so   by Theorem \ref{2.1.2}, $(\alpha,\delta(Q))$ is a minimal pair for $w,$ but   $\delta(Q)=\bar{v}(\theta-\alpha)\in\Gamma_{\bar{v}},$  therefore $(\alpha,\bar{v}(\theta-\alpha))$ is a $(K,v)$-minimal pair. Hence  $\bar{v}(\alpha-\eta)<\bar{v}(\theta-\alpha),$ which by strong triangle law implies that 
$$\bar{v}(\theta-\eta)=\bar{v}(\alpha-\eta)<\bar{v}(\theta-\alpha).$$
Thus $(\theta,\alpha)$ is a $(K,v)$-distinguished pair. 
\end{proof}



 \begin{proof}[Proof of Corollary \ref{1.2.10}]
Suppose first that $Q$ has a saturated distinguished chain $(Q=Q_r, Q_{r-1},\ldots,Q_0)$  of ABKPs for $w.$  Then for each $i,$  $1\leq i\leq r,$  $\deg Q_{i-1}<\deg Q_i$  and since  $Q_i$ is an ABKP for $w,$ so $\epsilon(Q_{i-1})<\epsilon(Q_i),$ which in view of Proposition \ref{2.1.6} (iii), implies that $w_{Q_{i-1}}(Q_i)< w({Q_i}),$ i.e., $w_{Q_{i-1}}<w.$  Hence by Theorem \ref{1.2.8},  we have that  $Q_{i}\in \psi(Q_{i-1})$ for every $1\leq i\leq r.$

 Conversely,  
as $Q_{i}\in \psi(Q_{i-1}),$ so  $w_{Q_{i-1}}<w,$ which together with    $\deg Q_{i-1}<\deg Q_{i},$ in view of  Theorem \ref{1.2.8}, implies   that $(Q_{i},Q_{i-1})$ is a distinguished pair, for each $i,$  $1\leq i\leq r.$   Since $\deg Q_0=1,$  therefore $(Q=Q_r, Q_{r-1},\ldots,Q_0)$ is a saturated distinguished chain for $Q.$
\end{proof}

\begin{proof}[Proof of Theorem \ref{1.2.4}]
Suppose first that $w$ is a valuation-transcendental extension of $v$ to $K(X).$  Then from Theorem \ref{2.1.13}, for the complete set $\{Q_i\}_{i\in\Delta}$ of ABKPs for $w,$ $\Delta$ has a maximal element, (say)  $N,$ and   $w=w_{Q_N}.$ The  hypothesis that  $\vartheta_j=\emptyset,$  for every $0\leq j\leq N,$  together with  Remark \ref{1.2.12} (i),  implies  that $\Delta_j=\{j\},$  for every $0\leq j\leq N,$ i.e., $\{Q_0,Q_1,\ldots,Q_N\}$ is the complete set of ABKPs for $w. $ Now by  Remark \ref{1.2.12} (ii),  (iv), (v) and (vi), we have that $\deg Q_0=1,$   $\deg Q_{j-1}<\deg Q_{j},$  $\epsilon(Q_{j-1})<\epsilon(Q_{j}),$   and     $Q_{j}\in\psi(Q_{j-1}),$  for every $1\leq j\leq N.$   Hence
 from Corollary \ref{1.2.10}, it follows that $(Q_{N}, Q_{N-1},\ldots, Q_{0})$ is a saturated distinguished chain for $Q_N$ of length $N.$

Suppose now that, the extension $w$ is valuation-algebraic,  then for the complete set  $\{Q_i\}_{i\in\Delta}$ of ABKPs for $w,$   $\Delta$ has no maximal element (see Theorem \ref{2.1.13}) which in view of the  hypothesis implies that   $\Delta=\mathbb{N}\cup\{0\}.$   Again on using  Remark \ref{1.2.12} (ii),  (iv), (v) and (vi), we  get that    $\deg Q_0=1,$   $\deg Q_{i-1}<\deg Q_{i},$  $\epsilon(Q_{i-1})<\epsilon(Q_{i})$ and   $Q_{i}\in\psi(Q_{i-1}),$ for every $i\in\Delta.$  Thus from   Corollary \ref{1.2.10},   $(Q_{n},Q_{n -1},\ldots,Q_0)$ is a saturated distinguished chain for $Q_{n}$ of length $n,$ for every $n\in\mathbb{N}.$
\end{proof}
\begin{remark}
It is known that (cf. Proposition 5.1 and Corollary 5.2,  \cite{S}), for a rank $1$ valued field $(K,v)$ if $K$ is defectless or char $k_{v}=0,$ then in a complete set  $\{Q_i\}_{i\in\Delta}$ of ABKPs for $w,$ the set $\vartheta_j=\emptyset$ for each $j\in I.$  Moreover, in view of Theorem \ref{2.1.13},  $I\neq \{0\}$ whenever  $w$ is either valuation-algebraic or  is defined by some ABKP of degree atleast two.
\end{remark}

\section*{Acknowledgement}
Research of the first author is supported by CSIR (grant no.  09/045(1747)/2019-EMR-I).

\bibliographystyle{amsplain}

\end{document}